\documentclass{amsart}
\usepackage[utf8]{inputenc}
\usepackage{amsmath,amsfonts,latexsym,amssymb}

\newtheorem{thm}{Theorem}
\newtheorem{lem}{Lemma}

\newtheorem{proposition}{Proposition}

\theoremstyle{remark}

\newtheorem*{ack}{Acknowledgements}

\begin{document}

\markboth{Ritabrata Munshi}{Twists of $GL(3)$ $L$-functions}
\title[Twists of $GL(3)$ $L$-functions]{Twists of $GL(3)$ $L$-functions}

\author{Ritabrata Munshi}   

\address{School of Mathematics, Tata Institute of Fundamental Research, 1 Dr. Homi Bhabha Road, Colaba, Mumbai 400005, India.}  
\curraddr{Statistics and Mathematics Unit, Indian Statistical Institute, 203 B.T. Road, Kolkata 700108, India}   
\email{rmunshi@math.tifr.res.in}
\thanks{The author is supported by SwarnaJayanti Fellowship, 2011-12, DST, Government of India.}

\begin{abstract}
Let $\pi$ be a $SL(3,\mathbb Z)$ Hecke-Maass cusp form, and let $\chi$ be a primitive Dirichlet character modulo $M$, which we assume to be prime. In this note we revisit the subconvexity problem addressed in `The circle method and bounds for $L$-functions IV' and establish the following unconditional bound
\begin{align*}
L\left(\tfrac{1}{2},\pi\otimes\chi\right)\ll M^{3/4-1/308+\varepsilon}.
\end{align*}
\end{abstract}


\maketitle


\section{Introduction}
\label{introd}

In this note we return to the subconvexity problem addressed in \cite{Mu5}. Our aim here is to present an argument, a variation of the $GL(2)$ delta method technique introduced in \cite{Mu5}, which is more transparent and technically much simpler. As an advantage we are now able to write down an explicit subconvex exponent. But most importantly the present argument, unlike \cite{Mu5}, does not rely on the Ramanujan conjecture. We will prove the following unconditional subconvexity result. (Note that the exponent $3/4+\varepsilon$ is the convexity bound.)\\

\begin{thm}
\label{mthm}
Suppose $\pi$ is a $SL(3,\mathbb{Z})$ Hecke-Maass cusp form, and $\chi$ is a primitive Dirichlet character modulo $M$ (which we assume to be prime). Then we have
\begin{align}
\label{main-bound}
L\left(\tfrac{1}{2},\pi\otimes\chi\right)\ll M^{3/4-1/308+\varepsilon}.
\end{align}
\end{thm}
  
\bigskip
  
In general we will stick to the notations used in \cite{Mu5}. The reader may refer to that paper for a broader introduction to the problem and for basic definitions (also see \cite{G}). Here we start by recalling the $GL(2)$ delta method. Let $p$ be a prime number and let $k\equiv 3\bmod{4}$ be a positive integer (which will be of the size $1/\varepsilon$). Let $\psi$ be a character of $\mathbb{F}_p^\times$ satisfying $\psi(-1)=-1=(-1)^k$. We consider $\psi$ as a character modulo $pM$. 
The main novelty in this note is the use of the space $S_k(pM,\psi)$, in place of $S_k(p,\psi)$, for the $GL(2)$ $\delta$-method. The inclusion of $M$, which is the conductor of $\chi$, in the level is an analogue of the `congruence-equation trick' which was used in \cite{Mu4} in the context of the usual delta method. (This trick has turned out to be useful in other problems as well.) 
Let $H^\star(pM,\psi)$ be the set of newforms and we extend it to $H_k(pM,\psi)$ - an orthogonal Hecke basis of the space of cusp forms $S_k(pM,\psi)$. 
Let $P$ be a parameter and let 
$$
P^\star=\sum_{\substack{P<p<2P\\p\;\text{prime}}}\;\sum_{\psi\bmod{p}}\left(1-\psi(-1)\right).
$$
Using the Petersson formula we derive
\begin{align}
\label{circ-meth}
&\delta(n,r)=\frac{1}{P^\star}\:\sum_{\substack{P<p<2P\\p\;\text{prime}}}\;\sum_{\psi\bmod{p}}\left(1-\psi(-1)\right)\sum_{f\in H_k(pM,\psi)}\omega_f^{-1}\lambda_f(n)\overline{\lambda_f(r)}\\
\nonumber &-\frac{2\pi i}{P^\star}\sum_{\substack{P<p<2P\\p\;\text{prime}}}\;\sum_{c=1}^\infty \frac{1}{cpM}\sum_{\psi\bmod{p}}\left(1-\psi(-1)\right)S_\psi(r,n;cpM)J_{k-1}\left(\frac{4\pi\sqrt{nr}}{cpM}\right),
\end{align}
i.e. the right hand side is $1$ if $n=r$, and is equal to $0$ otherwise.\\

Let $\mathcal{L}$ be a set of primes in the range $[L,2L]$, with $|\mathcal{L}|=L^\star\gg L^{1-\varepsilon}$ and $L\ll M^{1-\varepsilon}$. Consider the sum
\begin{align}
\label{F}
\mathcal{F}=&\frac{1}{L^\star P^\star}\:\sum_{\substack{P<p<2P\\p\;\text{prime}}}\;\sum_{\psi\bmod{p}}\left(1-\psi(-1)\right)\sum_{f\in H_k(pM,\psi)}\omega_f^{-1}\\
\nonumber\times &\sum_{\ell\in\mathcal{L}}\bar{\chi}(\ell)\mathop{\sum\sum}_{m,n=1}^\infty \lambda(m,n)\lambda_f(n\ell)W\left(\frac{nm^2}{N}\right)\sum_{r=1}^\infty \overline{\lambda_f(r)}\chi(r)V\left(\frac{r}{N\ell}\right).
\end{align}
Using the Petersson formula we see that the diagonal term is exactly the sum of interest 
and 
$$
S^\star(N)=\mathop{\sum\sum}_{m,n=1}^\infty \lambda(m,n)\chi(n)W\left(\frac{nm^2}{N}\right)V\left(\frac{n}{N}\right).
$$ 
Here $W$ is a smooth bump function with support $[1,2]$ and $V$ is a smooth function supported in $[M^{-4\theta},4]$, with $V(x)=1$ for $x\in [2M^{-4\theta},2]$, and satisfying $y^jV^{(j)}(y)\ll_j 1$. In Lemma~6 of Section 2 of \cite{Mu5} we showed that
\begin{align}
\label{l-s-bd}
L(\tfrac{1}{2},\pi\otimes\chi)\ll M^\varepsilon\sup_N \:\frac{|S^\star(N)|}{\sqrt{N}}+ M^{3/4-\theta/2+\varepsilon},
\end{align}
where the supremum is taken over $N$ in the range $M^{3/2-\theta}<N<M^{3/2+\theta}$. \\

In fact from \eqref{circ-meth} it follows that
$
S^\star(N)=\mathcal{F}-2\pi i\:\mathcal{O}
$ 
where the off-diagonal is given by
\begin{align}
\label{off-diag-1}
\mathcal{O}=&\frac{1}{L^\star P^\star}\sum_{\substack{P<p<2P\\p\;\text{prime}}}\;\sum_{\psi\bmod{p}}\left(1-\psi(-1)\right) \mathop{\sum\sum}_{m,n=1}^\infty \lambda(m,n)W\left(\frac{nm^2}{N}\right)\\
\nonumber &\times \sum_{\ell\in\mathcal{L}}\sum_{r=1}^\infty \chi(r\bar{\ell})V\left(\frac{r}{N\ell}\right)\sum_{c=1}^\infty \frac{S_\psi(r,n\ell;cpM)}{cpM}J_{k-1}\left(\frac{4\pi\sqrt{n\ell r}}{cpM}\right).
\end{align}\\

The set of primes $\mathcal{L}$ plays a subtle role. This has been introduced to split the modulus, at one stage, as a product of two numbers. This is not completely apparent in the beginning as it takes up the role of a modulus only after an application of reciprocity. The splitting of the modulus is used at the last application of Cauchy inequality followed by the Poisson summation. One will notice that putting the whole modulus inside makes the modulus for Poisson too large for any saving, whereas putting the whole modulus outside makes the diagonal too small. The situation is similar to one that we faced in \cite{Mu}, where we employed Jutila's version of the circle method to split the modulus. \\

The main result will follow from the following two propositions.\\

\begin{proposition} 
\label{prop1}
Let $\mathcal{O}$ be as defined in \eqref{off-diag-1}. Suppose $L<P$ and $\theta<1/2$. We have
\begin{align*}
\mathcal{O}\ll N^{1/2}M^\varepsilon\left\{\frac{M^{3/4+3\theta/2}L}{P}+\frac{M^{1+\theta}}{P}\right\}.
\end{align*} 
\end{proposition}

\bigskip

\begin{proposition}
\label{prop2}
If $M^{2\theta+\varepsilon}\ll L<M^2/N$, we have
\begin{align*}
\mathcal{F}\ll M^\varepsilon N^{1/2}\left[P^{1/2}M^{1/2+9\theta}+ \frac{P^{1/4}M^{5/8+17\theta/4}}{L^{1/2}}\left(L^{3/4}+\frac{M^{11\theta/4}P^{3/4}}{M^{1/8}}\right)\right].
\end{align*}
\end{proposition}

\bigskip

Indeed plugging the bounds from the above propositions we obtain
\begin{align*}
L(\tfrac{1}{2},\pi\otimes\chi)\ll & M^\varepsilon\left\{P^{1/2}M^{1/2+9\theta}+ \frac{P^{1/4}M^{5/8+17\theta/4}}{L^{1/2}}\left(L^{3/4}+\frac{M^{11\theta/4}P^{3/4}}{M^{1/8}}\right)\right\}\\
&+M^{\varepsilon}\left\{\frac{M^{3/4+3\theta/2}L}{P}+\frac{M^{1+\theta}}{P}\right\}+ M^{3/4-\theta/2+\varepsilon}.
\end{align*}
Then we optimally choose the three parameters - $P$, $L$ and $\theta$. It turns out that we will have $M^{4\theta}L\ll P$. So that the bound in the above corollary reduces to
\begin{align*}
\frac{P^{1/4}M^{5/8+17\theta/4+\varepsilon}}{L^{1/2}}\left(L^{3/4}+\frac{M^{11\theta/4}P^{3/4}}{M^{1/8}}\right)+\frac{M^{1+\theta+\varepsilon}}{P}+ M^{3/4-\theta/2+\varepsilon}.
\end{align*}
The optimum choice of $L$ is obtained by equating the first two terms. This gives $L=PM^{-1/6+11\theta/3}$, and reduces the above bound to
\begin{align*}
M^{7/12+31\theta/6+\varepsilon}P^{1/2}+\frac{M^{1+\theta+\varepsilon}}{P}+ M^{3/4-\theta/2+\varepsilon}.
\end{align*}
Equating the first two terms we now get the optimum choice for $P$, which turns put to be $P=M^{5/18-25\theta/9}$. Ultimately we find that the optimum choice of $\theta$ is given by $\theta=1/154$. This completes the proof of the Theorem.\\

\textbf{Notation:} Suppose $\mathcal{A}\ll M^\varepsilon\sum_{b\in\mathcal{F}}|\mathcal{B}_b|+M^{-2016}$ where $|\mathcal{F}|\ll M^\varepsilon$ and the implied constants depend only on $\varepsilon$. Then we write
$$
\mathcal{A}\lhd_\mathcal{F} \mathcal{B}_b,\;\;\;\;\text{or simply as}\;\;\;\;\mathcal{A}\lhd \mathcal{B},
$$
where there is no scope of confusion. \\

\ack
The author wishes to thank Roman Holowinsky and Zhi Qi for several helpful discussions related to the method presented in this paper. He thanks Qi for pointing out the cancellation of the oscillatory factor of the Bessel function, which is used in the proof of Lemma~\ref{observe}.

\bigskip

\section{Outline of the proof}

In this section we give a brief outline of the proof. Let $N=M^{3/2}$ and temporarily assume the Ramanujan conjecture $|\lambda(m,n)|\ll (mn)^\varepsilon$. First consider the off-diagonal term as given in \eqref{off-diag-1}. For convenience assume that $m=1$, and that $c$ is in the transition range, i.e. $c\sim C=NL/PM=M^{1/2}L/P$. Consider the generic case $p\nmid c$. The character sum 
\begin{align*}
\sum_{\psi\bmod{p}}\left(1-\psi(-1)\right) S_\psi(r,n\ell;cpM)
\end{align*}
can be partially evaluated, and one gets
\begin{align*}
pS(\bar{p}r,\bar{p}n\ell; cM)e\left(\pm\frac{\overline{cM}(r+n\ell)}{p}\right).
\end{align*}
Applying the reciprocity relation we see that the sum \eqref{off-diag-1} is essentially given by
\begin{align*}
\mathcal{O}\approx &\frac{1}{NL^2P}\sum_{\substack{P<p<2P\\p\;\text{prime}}}\; \mathop{\sum}_{n\sim N} \lambda(1,n)\: \sum_{\ell\in\mathcal{L}}\sum_{r\sim NL} \chi(r\bar{\ell})\sum_{c\sim C} S(\bar{p}r,\bar{p}n\ell; cM)e\left(\frac{\overline{p}(r+n\ell)}{cM}\right).
\end{align*}
Notice the presence of $M$ in the modulus $cM$. This acts as a conductor lowering trick as in \cite{Mu4}. \\

Assuming the Weil bound, we see that the Petersson formula gives a saving of size $\sqrt{PM}/\sqrt{C}$, and in addition we have saved $\sqrt{P}$ in the sum over $\psi$. 
Next we apply the Poisson summation formula on the $r$ sum. Since the length of the sum $r\sim NL$ is larger than the modulus $cM\sim NL/P$, we are only left with the zero frequency. Hence from Poisson we save $\sqrt{CM}$. Our initial target was to save $NL$, and so far we have saved $PM$. So now it remains to save $M^{1/2}L/P$, in the sum
\begin{align*}
\sum_{\substack{P<p<2P\\p\;\text{prime}}}\; \mathop{\sum}_{n\sim N} \lambda(1,n)\: \sum_{\ell\in\mathcal{L}}\sum_{c\sim C} \chi(pc\bar{\ell}) \mathfrak{D}(\bar{p}\bar{c}n\ell; M)
\end{align*}
where the character sum $\mathfrak{D}$ is as given in \eqref{d-char-sum}. We apply the Cauchy inequality and reduce the problem to that of saving $ML^2/P^2$ in the sum
\begin{align*}
\mathop{\sum}_{n\sim N}\Bigl|\sum_{\substack{P<p<2P\\p\;\text{prime}}}\;  \sum_{\ell\in\mathcal{L}}\sum_{c\sim C} \chi(pc\bar{\ell}) \mathfrak{D}(\bar{p}\bar{c}n\ell; M)\Bigr|^2.
\end{align*}
Next we open the absolute square (after smoothing) and apply the Poisson summation on the sum over $n$ with modulus $M$. Only the zero frequency survives. In the diagonal we at most save $PLC\sim M^{1/2}L^2$ (which will be smaller than the modulus $M$) and in the off-diagonal we save $M$. Crucially the structure of the character sum $\mathfrak{D}$ is such that for the zero-th frequency the saving is the full modulus $M$ and not just the square-root of the modulus. It  turns out that the off-diagonal $\mathcal{O}$ is fine if $P\gg \max\{L,M^{1/4}\}$. This is the content of Proposition~\ref{prop1}. Note that we do not require to utilize the oscillation in the  Fourier coefficients $\lambda(1,n)$ (as in \cite{Mu5}).\\

Next we consider the sum $\mathcal{F}$ as given in \eqref{F}. This is essentially given by
\begin{align*}
\frac{1}{L P^2}\:\sum_{\substack{P<p<2P\\p\;\text{prime}}}\;\sum_{\psi\bmod{p}}\left(1-\psi(-1)\right)\sum_{f\in H_k(pM,\psi)}\omega_f^{-1}\:\sum_{\ell\in\mathcal{L}}\bar{\chi}(\ell)\lambda_f(\ell)\mathop{\sum}_{n\sim N} \lambda(1,n)\lambda_f(n)\sum_{r\sim NL}\overline{\lambda_f(r)}\chi(r).
\end{align*}
We apply functional equation to the $n$ sum and ($GL(2)$) Voronoi summation to the $r$ sum. We save $N/(P^3M^3)^{1/2}$ in the $GL(3)\times GL(2)$ functional equation and save $NL/MP^{1/2}$ in the Voronoi summation. As initially we needed to save $NL$, it follows that we now need to save $MP^2$ in the sum
\begin{align*}
\sum_{\substack{P<p<2P\\p\;\text{prime}}}\;&\sum_{\psi\bmod{p}}\chi(p)\left(1-\psi(-1)\right)\psi(M)g_\psi^2\;\sum_{f\in H_k(pM,\psi)}\omega_f^{-1}\\
&\times \mathop{\sum}_{n\sim NP^3} \lambda(1,n)\bar{\lambda}_f(p^2Mn)\:\sum_{\ell\in\mathcal{L}}\sum_{r\sim M^{1/2}P/L} \lambda_f(r\ell)\bar{\chi}(r\ell).
\end{align*}
Observe that $r$ and $\ell$ occur together and it appears as if we have split the dual variable into a product of two variables, whose size we can regulate by choosing $L$. We apply the Petersson formula to arrive at
\begin{align*}
\sum_{\substack{P<p<2P\\p\;\text{prime}}}\;&\sum_{\psi\bmod{p}}\chi(p)\left(1-\psi(-1)\right)\psi(M)g_\psi^2\\
&\times \mathop{\sum}_{n\sim NP^3} \lambda(1,n)\:\sum_{\ell\in\mathcal{L}}\sum_{r\sim M^{1/2}P/L} \bar{\chi}(r\ell)\sum_{c\sim C}S_\psi(p^2Mn,r\ell;cpM),
\end{align*}
where the new transition range is given by $C=M^{1/2}P^2$. We have saved $\sqrt{PM}/\sqrt{C}$ (assuming Weil) from Petersson and we now need to save $M^{3/4}P^{5/2}$. The Kloosterman sum splits as 
$$
S(Mn,r\ell;cM)S_\psi(0,r\ell\overline{cM};p).
$$
The $\psi$ sum now gives a saving of size $P^{1/2}$ and the $GL(3)$ Voronoi summation gives a saving of size $M^{3/4}$. (Notice that the modulus is just $c$.) Also the $GL(3)$ Voronoi transforms the Kloosterman sum into an additive character. Our job reduces to saving $P^2$ in the sum 
\begin{align*}
\sum_{\substack{P<p<2P\\p\;\text{prime}}}\;\chi(p)\sum_{\ell\in\mathcal{L}}\;\sum_{r\sim M^{1/2}P/L}\;\mathop{\sum}_{\substack{c\sim C}}\bar{\chi}(r\ell)e\left(\frac{c\overline{r\ell}}{p}\right)\;\sum_{n\sim P^3} \lambda(n,1)e\left(-\frac{\overline{r\ell}Mn}{c}\right).
\end{align*}
Applying reciprocity we obtain
\begin{align*}
\sum_{\substack{P<p<2P\\p\;\text{prime}}}\;\chi(p)\sum_{\ell\in\mathcal{L}}\;\sum_{r\sim M^{1/2}P/L}\;\mathop{\sum}_{\substack{c\sim C}}\bar{\chi}(r\ell)e\left(-\frac{c\overline{p}}{r\ell}\right)\;\sum_{n\sim P^3} \lambda(n,1)e\left(\frac{\overline{c}Mn}{r\ell}\right).
\end{align*}
We can now apply the Poisson summation on the sum over $c$. The length of the sum is $M^{1/2}P^2$ and the modulus is $r\ell\sim M^{1/2}P$. So we are just left with the zero frequency and get a saving of $M^{1/4}P^{1/2}$. It remains to save $P^{3/2}/M^{1/4}$ in the sum
\begin{align*}
\sum_{\substack{P<p<2P\\p\;\text{prime}}}\;\chi(p)\sum_{\ell\in\mathcal{L}}\mathop{\sum}_{\substack{r\sim M^{1/2}P/L}}\bar{\chi}(r\ell)\;\sum_{n\sim P^3} \lambda(n,1)S(Mn,-\bar{p};r\ell).
\end{align*}
Applying Cauchy we see that we need to save $P^{3}/M^{1/2}$ in the sum
\begin{align*}
\sum_{n\sim P^3}\mathop{\sum}_{\substack{r\sim M^{1/2}P/L}}\Bigl|\sum_{\substack{P<p<2P\\p\;\text{prime}}}\;\sum_{\ell\in\mathcal{L}}\chi(p\bar{\ell})\; \lambda(n,1)S(Mn,-\bar{p};r\ell)\Bigr|^2.
\end{align*}
The diagonal is fine if $PL>P^3/M^{1/2}$ or $L>P^2/M^{1/2}$. In the off-diagonal we save $P^3/M^{1/4}(PL)^{1/2}$, which is enough if $L<M^{1/2}/P$. This is the content of Proposition~\ref{prop2}. In particular we have a choice for $L$ as long as $P<M^{1/3}$. Recall that the off-diagonal $\mathcal{O}$ was fine if $P>M^{1/4}$. Consequently we have a choice for the parameters $P$ and $L$ such that we have subconvex bounds for both $\mathcal{F}$ and $\mathcal{O}$.\\

\section{The off-diagonal}

 The off-diagonal contribution $\mathcal{O}$, is analyzed in the same spirit as Section 4 of \cite{Mu5}. After executing the $\psi$ sum we apply the reciprocity relation and then the Poisson summation on the sum over $r$. Next we get rid of the Fourier coefficients using the Cauchy inequality, and apply the Poisson summation on the sum over $n$.  For notational convenience we will only consider the subsum of \eqref{off-diag-1} where $p\nmid c$, which will be denoted by $\mathcal{O}_0$. The other case can be tackled in the same fashion, and we at the end get a stronger bound for that sum. Let 
\begin{align}
\label{d-char-sum}
\mathfrak{D}(u;M)=\sum_{\substack{b\bmod{M}\\(b(b-1),M)=1}}\bar{\chi}(b-1)e\left(\frac{(\bar{b}-1)u}{M}\right),
\end{align}
and
\begin{align*}
\mathfrak{I}(n,p,\ell;cM)=\int_\mathbb{R} e\left(\frac{N\ell y+n\ell}{cpM}\right)J_{k-1}\left(\frac{4\pi\sqrt{Nn\ell^2 y}}{cpM}\right)V(y)\mathrm{d}y.
\end{align*}
Let $\mathcal{C}=NLM^\varepsilon/PM$.\\

\begin{lem}
\label{lem-31}
We have
\begin{align}
\label{o-not-1}
\mathcal{O}_0\;\lhd\; \frac{N}{P^2M^{3/2}}\sum_{\substack{P<p<2P\\p\;\text{prime}}}\;\mathop{\sum\sum}_{m,n=1}^\infty \lambda(m,n)W\left(\frac{nm^2}{N}\right)\; \sum_{\ell\in\mathcal{L}} \chi(\bar{\ell})\;\sum_{c\ll \mathcal{C}} \chi(pc)\:\frac{\mathfrak{D}(\overline{pc}n\ell;M)}{c}\mathfrak{I}(n,p,\ell;cM).
\end{align} 
\end{lem}

\begin{proof}
Consider the sum in \eqref{off-diag-1} with $p\nmid c$ (the generic case).
The Bessel function is negligibly small if $c\gg \mathcal{C}$. Hence we only consider $c\ll \mathcal{C}$. (So $(c,M)=1$.)
As $p\nmid c$, the character sum 
\begin{align*}
\sum_{\psi\bmod{p}}\left(1-\psi(-1)\right) S_\psi(r,n\ell;cpM)
\end{align*}
can be replaced by
\begin{align*}
pS(\bar{p}r,\bar{p}n\ell; cM)e\left(\pm\frac{\overline{cM}(r+n\ell)}{p}\right).
\end{align*}
Consider the plus term. We apply the reciprocity relation
\begin{align*}
e\left(\frac{\overline{cM}(r+n\ell)}{p}\right)=e\left(-\frac{\overline{p}(r+n\ell)}{cM}\right)e\left(\frac{(r+n\ell)}{cpM}\right),
\end{align*}
and then push the last term to the weight function. \\

Next we apply the Poisson summation formula on the sum over $r$ with modulus $cM$. It turns out that the nonzero frequencies make a negligible contribution. Hence the sum
\begin{align*}
\sum_{r=1}^\infty \chi(r)S(\bar{p}r,\bar{p}n\ell; cM)e\left(-\frac{\overline{p}(r+n\ell)}{cM}\right)e\left(\frac{(r+n\ell)}{cpM}\right)V\left(\frac{r}{N\ell}\right)J_{k-1}\left(\frac{4\pi\sqrt{n\ell r}}{cpM}\right),
\end{align*}
upto a negligible error term, 
reduces to
\begin{align*}
\frac{N\ell}{cM}\mathfrak{C}(n,p,\ell;cM)\:\mathfrak{I}(n,p,\ell;cM)
\end{align*}
where the character sum is given by
\begin{align*}
\mathfrak{C}(n,p,\ell;cM)=\sum_{a\bmod{cM}}\:\chi(a)S(\bar{p}a,\bar{p}n\ell; cM)e\left(-\frac{\overline{p}(a+n\ell)}{cM}\right)
\end{align*}
and $\mathfrak{I}(n,p,\ell;cM)$ as above.
Using the standard bound for the Bessel function and the second derivative bound we get that $\mathfrak{I}(\dots)\ll cPMm/NL$. It follows that we have
\begin{align*}
\mathcal{O}_0\;\lhd\;&\frac{N}{P^2M^2}\sum_{\substack{P<p<2P\\p\;\text{prime}}}\;\mathop{\sum\sum}_{m,n=1}^\infty \lambda(m,n)W\left(\frac{nm^2}{N}\right)\; \sum_{\ell\in\mathcal{L}} \chi(\bar{\ell})\;\sum_{c\ll \mathcal{C}} \frac{\mathfrak{C}(n,p,\ell;cM)}{c^2}\mathfrak{I}(n,p,\ell;cM).
\end{align*}\\

The character sum can be partially evaluated. First it decomposes as a product of two character sums. The one with modulus $c$ is given by
\begin{align*}
\mathfrak{C}_1=\sum_{a\bmod{c}}\:S(\bar{p}\bar{M} a,\bar{p}\bar{M} n\ell; c)e\left(-\frac{\overline{pM}(a+n\ell)}{c}\right)=c.
\end{align*}
The other one with modulus $M$ is given by
\begin{align*}
\mathfrak{C}_2&=\sum_{a\bmod{M}}\:\chi(a)S(\bar{p}\bar{c} a,\bar{p}\bar{c} n\ell; M)e\left(-\frac{\overline{pc}(a+n\ell)}{M}\right)\\
&=\sideset{}{^\star}\sum_{b\bmod{M}}e\left(\frac{(\bar{b}-1)\overline{pc}n\ell}{M}\right)\sum_{a\bmod{M}}\:\chi(a)e\left(\frac{(b-1)\overline{pc}a}{M}\right).
\end{align*}
The inner sum vanishes unless $(b-1,M)=1$, in which case it is given by $\chi(pc)\bar{\chi}(b-1)g_\chi$. One gets
\begin{align*}
\mathfrak{C}(\dots)=cg_\chi \chi(pc)\sum_{\substack{b\bmod{M}\\(b(b-1),M)=1}}\bar{\chi}(b-1)e\left(\frac{(\bar{b}-1)\overline{pc}n\ell}{M}\right).
\end{align*}
\end{proof}

Applying the Cauchy inequality to the right hand side of \eqref{o-not-1}, we arrive at
\begin{align}
\label{o-not}
\mathcal{O}_0\;\lhd\;&\frac{N}{P^2M^{3/2}}\;\mathfrak{V}^{1/2}\;\mathfrak{U}^{1/2},
\end{align}
where
\begin{align*}
\mathfrak{V}=\;\mathop{\sum\sum}_{nm^2\sim N} m\;|\lambda(m,n)|^2 
\end{align*}
and
\begin{align*}
\mathfrak{U}=\;\mathop{\sum\sum}_{nm^2\sim N} \frac{1}{m}\; \Bigl|\sum_{\substack{P<p<2P\\p\;\text{prime}}}\sum_{\ell\in\mathcal{L}}\sum_{c\ll \mathcal{C}}\frac{\chi(pc\bar{\ell})}{c} \;\mathfrak{D}(\overline{pc}n\ell;M)\mathfrak{I}(n,p,\ell;cM)\Bigr|^2.
\end{align*}

\begin{lem}
\label{uv-bd}
We have
\begin{align*}
\mathfrak{V} &\ll N^{1+\varepsilon}
\end{align*}
and
\begin{align*}
\mathfrak{U}\ll  NL^2P^2+M^2P^2.
\end{align*}
\end{lem}

\begin{proof}
The Hecke relation implies that $|\lambda(m,n)|^2\ll N^\varepsilon \sum_{d|(m,n)}|\lambda(m/d,1)|^2|\lambda(1,n/d)|^2$, and we have the Ramanujan bound on average
$$
\sum_{e\ll E}|\lambda(1,e)|^2\ll E^{1+\varepsilon}.
$$
It follows that
\begin{align*}
\mathfrak{V} &\ll N^\varepsilon\sum_{m\ll N^{1/2}}m\sum_{d|m}|\lambda(m/d,1)|^2 \;\mathop{\sum}_{n\ll N/m^2d} \;|\lambda(1,n)|^2 \\
&\ll N^{1+\varepsilon}\sum_{d\ll N^{1/2}}\frac{1}{d^2}\sum_{m\ll N^{1/2}/d}\:\frac{|\lambda(m,1)|^2}{m}\ll N^{1+\varepsilon}.
\end{align*}\\

Next we consider $\mathfrak{U}$. The Bessel function in the integral $\mathfrak{I}$ is negligibly small if $m\gg NLM^\varepsilon/PM=\mathcal{M}$. For $m\ll \mathcal{M}$, and $C\ll \mathcal{C}/m$ we consider
\begin{align*}
\mathfrak{U}_m=\;\mathop{\sum}_{n\in \mathbb{Z}} W\left(\frac{n}{\mathcal{N}}\right)\; \Bigl|\sum_{\substack{P<p<2P\\p\;\text{prime}}}\sum_{\ell\in\mathcal{L}}\sum_{c\sim C}\frac{\chi(pc\bar{\ell})}{c} \;\mathfrak{D}(\overline{pc}n\ell;M)\mathfrak{I}(n,p,\ell;cM)\Bigr|^2,
\end{align*}
where 
$$\mathcal{N}=\max\left\{N_0,\frac{NL}{CPm}\right\}M^\varepsilon,$$
and $W\in C^\infty (-1,1)$. Opening the absolute square we apply the Poisson summation on the outer sum with modulus $M$. Taking into account the possible oscillation in the weight function, we see that our choice of the extended length of the sum implies that the nonzero frequencies make a negligible contribution. It follows that
\begin{align*}
\mathfrak{U}_m=\;\frac{\mathcal{N}}{M}\mathop{\sum\sum}_{\substack{P<p,p'<2P\\p,p'\;\text{prime}}}\mathop{\sum\sum}_{\ell,\ell'\in\mathcal{L}}\mathop{\sum\sum}_{c,c'\sim C}\frac{\chi(pc\overline{p'c'}\bar{\ell}\ell')}{cc'} \;\mathfrak{C}_3\;\mathfrak{J}+O(M^{-2016})
\end{align*}
where the character sum is given by
\begin{align*}
\mathfrak{C}_3=\sum_{a\bmod{M}}\mathfrak{D}(\overline{pc}a\ell;M)\overline{\mathfrak{D}(\overline{p'c'}a\ell';M)}
\end{align*}
and the integral is given by
\begin{align*}
\mathfrak{J}=\int_\mathbb{R} \mathfrak{I}(\mathcal{N}y,p,\ell;cM)\overline{\mathfrak{I}(\mathcal{N}y,p',\ell';c'M)}W(y)\mathrm{d}y.
\end{align*}
We use the trivial bound $\mathfrak{J}\ll (CPMm)^2/(NL)^2$. Now consider the character sum which is given by
\begin{align*}
\mathfrak{C}_3=\mathop{\sum\sum}_{\substack{b,b'\bmod{M}\\(bb'(b-1)(b'-1),M)=1}}\bar{\chi}(b-1)\chi(b'-1)\sum_{a\bmod{M}} e\left(\frac{(\bar{b}-1)\overline{pc}a\ell-(\bar{b}'-1)\overline{p'c'}a\ell'}{M}\right).
\end{align*}
We get
\begin{align*}
\mathfrak{C}_3=M\mathop{\sum\sum}_{\substack{b,b'\bmod{M}\\(bb'(b-1)(b'-1),M)=1\\\bar{b}'\equiv 1+(\bar{b}-1)v\bmod{M}}}\bar{\chi}(b-1)\chi(b'-1)
\end{align*}
where $v\equiv \overline{pc}\ell p'c'\bar{\ell}'\bmod{M}$. Now
\begin{align*}
\overline{b-1}(\overline{1+(\bar{b}-1)v}-1)\equiv \overline{1+b(\bar{v}-1)}
\end{align*}
and consequently
\begin{align*}
\mathfrak{C}_3=M\mathop{\sum}_{\substack{b\bmod{M}\\(b(b-1),M)=1}}\bar{\chi}(1+b(\bar{v}-1))=\begin{cases} M(M-2) &\text{if}\;\;v\equiv 1\bmod{M}\\ O(M) &\text{otherwise}.
\end{cases}
\end{align*}
It follows that
\begin{align*}
\mathfrak{U}_m\ll \;\mathcal{N}\:\left(\frac{P^4M^2C^2m^2}{N^2}+\frac{P^3M^3Cm^2}{N^2L} \right),
\end{align*}
and we have
\begin{align*}
\mathfrak{U}\ll \sum_{m\ll\mathcal{M}}\;\frac{\mathcal{N}}{m}\:\left(\frac{P^4M^2C^2m^2}{N^2}+\frac{P^3M^3Cm^2}{N^2L} \right)\ll NL^2P^2+M^2P^2.
\end{align*}
In the last inequality we assumed that $L<P$ and $M<N$ (say $\theta<1/2$).
\end{proof}

Plugging in the bounds from Lemma~\ref{uv-bd} to \eqref{o-not} we get the bound of Proposition~\ref{prop1} for the sum $\mathcal{O}_0$. In the case $p|c$, we obtain a stronger bound by employing the above analysis. \\

\section{Treating the old forms}

In the rest of the paper we will prove Proposition~\ref{prop2}. To analyse the sum $\mathcal{F}$ we use  the functional equation for $GL(3)\times GL(2)$ Rankin-Selberg convolution (as in \cite{Mu5}). There are two new issues. First we need to split $\lambda_f(n\ell)$ using the Hecke relation and secondly we need to take care of the oldforms. To this end let $\mathcal{F}^\sharp$ be same as the expression in \eqref{F} with $\lambda_f(n\ell)$ replaced by $\lambda_f(n)\lambda_f(\ell)$.\\

\begin{lem}
Suppose $L\gg M^{2\theta+\varepsilon}$. We have
\begin{align*}
\mathcal{F}=\mathcal{F}^\sharp + O(M^{-2016}).
\end{align*}
\end{lem}

\begin{proof}
On using the Hecke relation we get two terms, one of them being $\mathcal{F}^\sharp$. To tackle the other term, consider
\begin{align}
\label{hr}
\sum_{f\in H_k(pM,\psi)}\omega_f^{-1}\;\sum_{\ell\in\mathcal{L}}\bar{\chi}(\ell)\mathop{\sum\sum}_{m,n=1}^\infty \lambda(m,n\ell)\lambda_f(n)W\left(\frac{n\ell m^2}{N}\right)\sum_{r=1}^\infty \overline{\lambda_f(r)}\chi(r)V\left(\frac{r}{N\ell}\right).
\end{align}
To this we apply the Petersson formula. Observe that the diagonal does not exist and the off-diagonal is given by
\begin{align*}
\sum_{c=1}^\infty\sum_{\ell\in\mathcal{L}}\bar{\chi}(\ell)\mathop{\sum\sum}_{m,n=1}^\infty \lambda(m,n\ell)W\left(\frac{n\ell m^2}{N}\right)\sum_{r=1}^\infty \chi(r)\frac{S_\psi(r,n;cpM)}{cpM}J_{k-1}\left(\frac{4\pi\sqrt{rn}}{cpM}\right)V\left(\frac{r}{N\ell}\right).
\end{align*}
For $c\gg NM^\varepsilon/pM$ the Bessel function is negligibly small. For smaller values of $c$, we apply the Poisson summation formula on the sum over $r$. As $L\gg M^{2\theta+\varepsilon}$, it follows that the non zero frequencies make a negligible contribution. For the zero frequency we have the character sum
\begin{align*}
\sum_{a\bmod{cpM}}\chi(a)S_\psi(a,n;cpM)
\end{align*}
which vanishes. More precisely the part modulo $p$-power vanishes. Hence the above sum \eqref{hr} is negligibly small. Consequently in \eqref{F} we can replace $\lambda_f(n\ell)$ by $\lambda_f(n)\lambda_f(\ell)$ at a cost of a negligible error term.
\end{proof}

Next we take into account the contribution of the old forms.\\

\begin{lem}
We have
\begin{align*}
\mathcal{F}=\mathcal{F}^\star +O(P^2M^{\varepsilon}),
\end{align*}
where
\begin{align}
\label{Fstar}
\mathcal{F}^\star=&\frac{1}{L^\star P^\star}\:\sum_{\substack{P<p<2P\\p\;\text{prime}}}\;\sum_{\psi\bmod{p}}\left(1-\psi(-1)\right)\sum_{f\in H_k^\star(pM,\psi)}\omega_f^{-1}\\
\nonumber\times &\sum_{\ell\in\mathcal{L}}\lambda_f(\ell)\bar{\chi}(\ell)\mathop{\sum\sum}_{m,n=1}^\infty \lambda(m,n)\lambda_f(n)W\left(\frac{nm^2}{N}\right)\sum_{r=1}^\infty \overline{\lambda_f(r)}\chi(r)V\left(\frac{r}{N\ell}\right).
\end{align}
\end{lem}

\begin{proof}
For $g\in S_k(p,\psi)$, we set $g|_M(z)=M^{k/2}g(Mz)$ which lies in $S_k(pM,\psi)$. Define 
$
g^\star=g|_M-\left<g|_M,g\right>\left<g,g\right>^{-1}g.
$
Then $\{g,g^\star: g\in H_k(p,\psi)\}$ gives an orthogonal Hecke basis of the space of oldforms. Note that $\lambda_{g|_M}(r)=0$ unless $M|r$, in which case $\chi(r)=0$. Also $\left<g|_M,g|_M\right>=\left<g,g\right>$, hence by Bessel inequality we have $|\left<g|_M,g\right>\left<g,g\right>^{-1}|\leq 1$. Consequently for $f=g$ or $g^\star$, functional equations yield the bound
\begin{align*}
\mathop{\sum\sum}_{m,n=1}^\infty \lambda(m,n)\lambda_f(n)W\left(\frac{nm^2}{N}\right)\sum_{r=1}^\infty \overline{\lambda_f(r)}\chi(r)V\left(\frac{r}{N\ell}\right)\ll p^{3/2}\;p^{1/2}M^{1+\varepsilon},
\end{align*}
as $\pi\otimes f$ has conductor $p^3$ and $f\otimes \chi$ has conductor $pM^2$. Hence the contribution of the oldforms is bounded by $p^2M^\varepsilon$. 
\end{proof}

\bigskip


\section{Applying functional equations}
\label{fe}

Next we will apply functional equation and $GL(2)$ Voronoi summation formula to the sums over $(m,n)$ and $r$. This will lead us to the family of dual sums
\begin{align}
\label{to-ana}
\mathcal{D}^\star=&\frac{N^2}{M^2P^5}\:\sum_{\substack{P<p<2P\\p\;\text{prime}}}\chi(p)\;\sum_{\psi\bmod{p}}\left(1-\psi(-1)\right)\psi(-M)g_{\psi}^2\sum_{f\in H_k^\star(pM,\psi)}\omega_f^{-1}\overline{\lambda_f(p^2M)}\\
\nonumber\times &\sum_{\ell\in\mathcal{L}}\lambda_f(\ell)\bar{\chi}(\ell)\mathop{\sum\sum}_{m,n=1}^\infty \lambda(m,n)\overline{\lambda_f(n)}W\left(\frac{nm^2}{\tilde{N}}\right)\sum_{r=1}^\infty \lambda_f(r)\bar{\chi}(r)V\left(\frac{r}{\tilde{R}}\right)
\end{align}
where $W$ is a bump function with support $[1,2]$, $V$ has the same support but $V^{(j)}\ll_j M^{2j\theta}$, and the lengths of the sums are determined by the following restrictions
\begin{align}
\label{range-all}
\frac{P^3M^3}{NM^{\varepsilon}}\ll \tilde{N}\ll \frac{P^3M^{3+\varepsilon}}{N}\;\;\;\text{and}\;\;\; \frac{M^{2}P}{NLM^\varepsilon}\ll \tilde{R}\ll \frac{M^{2+4\theta+\varepsilon} P}{NL}.
\end{align}
We can take $\tilde{R}$, $\tilde{N}$ to be dyadic, so that the size of the family is  $M^\varepsilon$. \\

\begin{lem}
\label{dualize}
We have
\begin{align*}
\mathcal{F}^\star\lhd \mathcal{D}^\star,
\end{align*}
where we are using the shorthand notation introduced in the Section~\ref{introd}.
\end{lem}

\begin{proof}
As in \cite{Mu5} we use the functional equation of $L(s,\pi\otimes f)$ to derive the following summation formula. Let $\mathcal{U}$ be a partition of unity and let $\mathcal{U}^\dagger$ be the subset of $\mathcal{U}$ consisting of those pairs $(U,\tilde{N})$ which have $\tilde{N}$ in the range
$
[P^3M^{3-\varepsilon}/N,P^3M^{3+\varepsilon}/N].
$
We have
\begin{align*}
\mathop{\sum\sum}_{m,n=1}^\infty &\lambda(m,n)\lambda_{f}(n)W\left(\frac{m^2n}{N}\right)=i^{3k}\psi(-M^2)g_{\psi}^3M^2\overline{\lambda(p^3M)} \sum_{\mathcal{U}^\dagger}\mathop{\sum\sum}_{m,n=1}^\infty \frac{\lambda(n,m)\overline{\lambda_f(n)}}{m^2n}U\left(\frac{m^2n}{\tilde N}\right)\\
&\times \frac{1}{2\pi i}\int_{(0)}\tilde{W}(s)\left(\frac{m^2nN}{p^3M^3}\right)^{s}\frac{\gamma(1-s)}{\gamma(s)} \mathrm{d}s + O(M^{-2015}).
\end{align*}
To the other sum
\begin{align*}
\sum_{r=1}^\infty \overline{\lambda_f(r)}\chi(r)V\left(\frac{r}{N\ell}\right)=g_{\bar{\chi}}^{-1}\sum_{a\bmod{M}}\bar{\chi}(a)\sum_{r=1}^\infty \overline{\lambda_f(r)}e\left(\frac{ar}{M}\right)V\left(\frac{r}{N\ell}\right)
\end{align*}
we apply the Voronoi summation formula. This transforms the above sum into
\begin{align*}
2\pi i^k\chi(-p)\bar{\psi}(-M)\frac{g_\chi g_{\bar{\psi}}}{g_{\bar{\chi}}}\frac{N\ell}{Mp}\sum_{r=1}^\infty \lambda_f(rp)\bar{\chi}(r)\int_0^\infty V(x)J_{k-1}\left(\frac{4\pi\sqrt{N\ell rx}}{M\sqrt{p}}\right)\mathrm{d}x.
\end{align*}
In the last integral $V$ is supported in $[M^{-4\theta},4]$, and satisfies $y^jV^{(j)}(y)\ll 1$. As $k$ is large the Bessel function is negligibly small if $r\ll M^2P/NLM^\varepsilon$. On the other hand making the change of variables $y^2=x$, pulling out the oscillation of the Bessel function and integrating by parts we get that the integral is negligibly small if $r\gg M^{2+4\theta+\varepsilon}P/NL$.
This reduces the analyses of the sum in \eqref{Fstar} to that of the sums of the type $\mathcal{D}^\star$. The lemma follows. (More details can be found in Section~5 of \cite{Mu5}.)
\end{proof}

If $f$ is an oldform coming from level $p$, the sub sum over $(m,n)$ in \eqref{to-ana}  is negligibly small. So the sum over $f$ can be extended to a complete Hecke basis at a cost of a negligible error term. Next we use the Hecke relation. We analyse the generic term. The other term can be analysed in the same fashion and at the end we get a stronger bound for it. Consider \eqref{to-ana} with $\lambda_f(\ell)\lambda_f(r)$ (resp. $\bar{\lambda}_f(p^2M)\bar{\lambda}_f(r)$) replaced by $\lambda_f(\ell r)$ (resp. $\bar{\lambda}_f(np^2M)$) and the sum over $f$ is extended to a full Hecke basis. We will denote this sum by $\mathcal{D}$. Using the Petersson formula this reduces to
\begin{align}
\label{123d}
\mathcal{D}=&\frac{N^2}{M^2P^5}\:\sum_{\substack{P<p<2P\\p\;\text{prime}}}\;\chi(p)\sum_{\psi\bmod{p}}\left(1-\psi(-1)\right)\psi(-M)g_{\psi}^2\sum_{c=1}^\infty \frac{1}{cpM}\\
\nonumber\times &\sum_{\ell\in\mathcal{L}}\;\sum_{r=1}^\infty\bar{\chi}( r\ell) V\left(\frac{r}{\tilde{R}}\right)\mathop{\sum\sum}_{m,n=1}^\infty \lambda(m,n)S_\psi(np^2M,r\ell;cpM) J_{k-1}\left(\frac{4\pi\sqrt{nr\ell}}{c\sqrt{M}}\right)W\left(\frac{nm^2}{\tilde{N}}\right) .
\end{align}
(The diagonal vanishes as $\chi(M)=0$.) The Kloosterman sum vanishes if $M|c$. On the other hand if $p|c$ then the sum vanishes unless $p|r$. So the contribution of these terms can be shown to be much smaller compared to the generic terms. For the generic term where $(r,p)=1$, we take dyadic subdivision of the $m$ sum and a smooth dyadic subdivision of the $c$ sum. Let the contribution of the block with $m\sim \mathfrak{m}$ and $c\sim C$ be denoted by $\mathcal{D}(C,\mathfrak{m})$. This sum is negligibly small if $C\gg \mathcal{C}=(\tilde{N}\tilde{R}L)^{1/2}M^\varepsilon/M^{1/2}\mathfrak{m}$.\\

\begin{lem}
For $C\ll \mathcal{C}$, we have
\begin{align*}
\mathcal{D}(C,\mathfrak{m})=&\frac{N^2}{CM^3P^4}\sum_{\substack{P<p<2P\\p\;\text{prime}}}\;\chi(p)\sum_{\ell\in\mathcal{L}}\mathop{\sum\sum}_{\substack{c,r=1\\(p,r)=1}}^\infty \bar{\chi}(r\ell)e\left(\frac{c\overline{r\ell}}{p}\right)\\
&\times \sum_{m\sim \mathfrak{m}}\mathop{\sum}_{n=1}^\infty\lambda(n,m)S(n,r\ell\bar{M};c) W\left(\frac{c}{C},\frac{n}{\tilde{N}_0},\frac{r}{\tilde{R}}\right) + O(M^{-2016})
\end{align*}
where 
$$
W(x,y,z)=J_{k-1}\left(\frac{4\pi\sqrt{\ell\tilde{N}_0\tilde{R}yz}}{\sqrt{M}Cx}
\right) x^{-1}W(x)W\left(y\right)V\left(z\right).
$$
Here $W$ and $V$ are as in \eqref{to-ana}.
\end{lem}

\begin{proof}
Observe that the Kloosterman sum in \eqref{123d} vanishes if $M|c$ as $(M,r\ell)=1$. Hence we get $$S_\psi(np^2M,r\ell;cpM)=-S_\psi(np^2,r\ell\bar{M};cp),$$ and for $p\nmid c$ it further reduces to $\psi(r\ell)\bar{\psi}(cM)g_{\bar{\psi}}S(n,r\ell\bar{M};c)$.  Then executing the sum over $\psi$, and taking a dyadic subdivision of the $m$ sum, smooth dyadic partition for the $c$-sum, we arrive at sums of the type given in the statement of the lemma.
\end{proof}

\bigskip

\section{Intertwining Voronoi and Poisson summations with reciprocity}

The next step involves an application of the Voronoi summation formula on the sum over $n$.\\ 

\begin{lem}
We can write $\mathcal{D}(C,\mathfrak{m})$ as a sum of two similar sums $\mathcal{D}^\pm(C,\mathfrak{m})$, where
\begin{align}
\label{dual-od-red-tran-a}
\mathcal{D}^+(C,\mathfrak{m})=&\frac{N^2C}{M^3P^4}\sum_{\substack{P<p<2P\\p\;\text{prime}}}\;\chi(p)\sum_{\ell\in\mathcal{L}}\mathop{\sum\sum}_{\substack{c,r=1\\(p,r)=1}}^\infty \bar{\chi}(r\ell)e\left(\frac{c\overline{r\ell}}{p}\right)\;\mathop{\sum}_{m\sim\mathfrak{m}}  \sum_{m'|cm}\sum_{n=1}^\infty \frac{\lambda(m',n)}{m'n}\\
\nonumber \times &\sum_{d|c}\frac{\mu(d)}{d}\sideset{}{^\star}\sum_{\substack{\beta\bmod{mc/m'}\\r\ell\bar{M}+\beta m'\equiv 0\bmod{c/d}}}e\left( \frac{\bar{\beta}n}{mc/m'}\right)W^\star_+\left(\frac{c}{C},\frac{m'^2n\tilde{N}_0}{c^3m},\frac{r}{\tilde{R}}\right).
\end{align}
\end{lem}

\begin{proof}
The proof follows in the same line as the analysis given in Section~7.1 of \cite{Mu5}
(see Lemma~19 and Lemma~20).
\end{proof}

The character sum can be evaluated quite easily. We write $cd$ in place of $c$, and then set $c_1=(m',c)$. Let $c=c_1c_2$, $m'=c_1m''$. Since the case where $\ell|c_1$, is much simpler compared to the generic situation $\ell\nmid c_1$, we will only provide the details for the generic case. The bound that we obtain in the other case is stronger.\\

\begin{lem}
Suppose $md/m''=q_1q_2$ with $(q_1,c_2)=1$ and $q_2|c_2^\infty$. Suppose $\ell\nmid c_1$. Then the character sum in \eqref{dual-od-red-tran-a} vanishes unless $c_1|r$ and $q_2|n$, in which case  we have 
\begin{align*}
\sideset{}{^\star}\sum_{\substack{\beta\bmod{mc/m'}\\r\ell\bar{M}+\beta m'\equiv 0\bmod{c/d}}}e\left( \frac{\bar{\beta}n}{mc/m'}\right)=q_2\mathfrak{c}_{q_1}(n)e\left(- \frac{\overline{r'\ell}m''Mn}{qc_2}\right),
\end{align*}
where $\mathfrak{c}_{q_1}(n)$ is the Ramanujan sum with modulus $q_1$ and $r'=r/c_1$.
\end{lem}

\begin{proof}
From the congruence condition it follows that $c_1$ necessarily divides $r\ell$. Since we are taking $\ell$ to be prime, and dealing with the generic case $\ell\nmid c_1$, it follows that $c_1|r$, and we write $r=c_1r'$. The congruence condition now yields 
$$
\bar{\beta}\equiv -\overline{r'\ell}m''M\bmod{c_2}.
$$
Hence
\begin{align*}
\sideset{}{^\star}\sum_{\substack{\beta\bmod{mc/m'}\\r\ell\bar{M}+\beta m'\equiv 0\bmod{c/d}}}e\left( \frac{\bar{\beta}n}{mc/m'}\right)=e\left(- \frac{\overline{r'\ell}m''Mn}{qc_2}\right)\sideset{}{^\dagger}\sum_{\substack{\delta\bmod{q}}}e\left( \frac{\delta n}{q}\right)
\end{align*}
where $q=md/m''$. Write $q=q_1q_2$ with $(q_1,c_2)=1$ and $q_2|c_2^\infty$. The dagger on the sum means that $(\delta,q_1)=1$. It follows that the sum vanishes unless $q_2|n$, in which case it reduces to $q_2\mathfrak{c}_{q_1}(n)$. The lemma follows.
\end{proof}

\bigskip


We set
\begin{align}
\label{dual-od-red-tran-b}
\mathfrak{D}= &\frac{N^2C}{M^2P^4(L\tilde{R})^2}\;\mathop{\sum\sum}_{\substack{m\sim\mathfrak{m}\\d\sim D}}\:\mathop{\sum\sum}_{c_1m''|dm}\:\mathop{\sum\sum}_{\substack{q_1q_2=dm/m''\\\tilde{q}_2\sim Q}}\:\frac{c_1m''q_2}{d^2m\tilde{q}_2}\;\Omega(\dots)
\end{align}
where
\begin{align}
\label{dual-od-red-tran-c}
\Omega(\dots)=&\mathop{\sum\sum}_{\substack{r'\sim \tilde{R}/c_1\\n\ll \mathfrak{N}}}\:|\mathfrak{c}_{q_1}(n)|\:|\lambda(c_1m'',q_2n)|\:\Bigl|\sum_{\substack{P<p<2P\\p\;\text{prime}}}\;\sum_{\ell\in\mathcal{L}}\mathop{\sum}_{\substack{c_2\in \mathcal{S}}} \chi(p\bar{\ell})\;S(c_2-\tilde{q}_2\bar{p},\:\bar{q}_1\bar{\tilde{q}}_2 m''Mn;r'\ell)\Bigr|.
\end{align}
Here  $\tilde{q}_2$ stands for the product of the prime factors of $q_2$, the set
$$
\mathcal{S}=\Bigl\{c_2\in\mathbb{Z}:
\Bigl|c_2-\frac{d\tilde{q}_2}{p}\Bigr|\ll \frac{DQ \tilde{R}L M^\varepsilon}{C}=: C_2\Bigr\}
$$
and
$$
\mathfrak{N}=\frac{(L\tilde{R})^{3/2}\tilde{N}^{1/2}M^\varepsilon}{q_2c_1^2m''^2M^{3/2}}.
$$\\

\begin{lem}
\label{observe}
We have
\begin{align}
\label{dual-od-red-tran-bb}
\mathcal{D}^+(C,\mathfrak{m})\;\lhd\; \mathfrak{D}.
\end{align}
\end{lem}

\begin{proof}
We use the reciprocity relation
\begin{align*}
e\left(-\frac{\overline{r'\ell}m''Mn}{qc_2}\right)=e\left(\frac{\overline{qc_2}m''Mn}{r'\ell}\right)e\left(-\frac{m''Mn}{qc_2r'\ell}\right)
\end{align*}
and push the last term to the weight function. A reminiscent of the fact - that the Bessel function is the archimedean analogue of the Kloosterman sum - is that the function
\begin{align*}
e\left(-\frac{m''Mn}{qc_2r'\ell}\right)W^\star_+\left(\frac{c_1c_2d}{C},\frac{m''^2n\tilde{N}_0}{c_1c_2^3d^3m},\frac{r}{\tilde{R}}\right)
\end{align*}
is `nice'. More precisely the above function may be replaced by
\begin{align*}
\frac{m'^2nM}{L\tilde{R}C\mathfrak{m}}\:V\left(\frac{c_1c_2d}{C}\right)V\left(\frac{m'^2n\tilde{N}_0}{C^3m}\right)V\left(\frac{r}{\tilde{R}}\right)
\end{align*}
where $V$ are bump functions, at a cost of introducing a family of size $M^\varepsilon$. The proof of this fact follows in the same line as the analysis given in Section~4 of \cite{L}.  Also observe that
$$
e\left(\frac{c\overline{r\ell}}{p}\right)=e\left(\frac{c_2d\overline{r'\ell}}{p}\right)=e\left(-\frac{c_2d\overline{p}}{r'\ell}\right)e\left(\frac{c_2d}{pr'\ell}\right)
$$
where the last factor is only mildly oscillating and we are going to absorb it in the weight function.\\

We are now ready to apply the Poisson summation formula on the sum over $c_2$.  Observe that since $q_2|c_2^\infty$ we have $\tilde{q}_2|c_2$ where $\tilde{q}_2$ stands for the product of the prime factors of $q_2$. We write $nq_2$ in place of $n$, and $\tilde{q}_2c_2$ in place of $c_2$. The Poisson summation transforms the sum 
\begin{align*}
\sum_{c_2\in\mathbb{Z}} e\left(\frac{\overline{q_1\tilde{q}_2 c_2}m''Mn-\tilde{q}_2 c_2\bar{p}}{r'\ell}\right)\:V\left(\frac{c_1c_2d\tilde{q}_2}{C}\right)e\left(\frac{c_2d\tilde{q}_2}{pr'\ell}\right)
\end{align*}
to
\begin{align*}
\frac{C}{c_1d\tilde{q}_2 r'\ell}\sum_{c_2\in\mathbb{Z}} S(c_2-\tilde{q}_2\bar{p},\:\bar{q}_1\bar{\tilde{q}}_2 m''Mn;r'\ell)\int_\mathbb{R}V(y)e\left(\frac{Cy}{c_1pr'\ell}-\frac{c_2Cy}{c_1d\tilde{q}_2 r'\ell}\right)\mathrm{d}y.
\end{align*}
Here $\tilde{q}_2$ stands for the content of $q_2$.
The sum over $c_2$ can be truncated at 
$$
\Bigl|c_2-\frac{d\tilde{q}_2}{p}\Bigr|\ll \frac{c_1d\tilde{q}_2 r'\ell M^\varepsilon}{C}
$$ at a cost of a negligible error term. The lemma now follows by getting rid of the weight function by introducing a family of sums of the type \eqref{dual-od-red-tran-c}.
\end{proof}

\section{Cauchy and Poisson}

The next lemma is a consequence of the Weil bound for the Kloosterman sum.\\

\begin{lem}
\label{lemma-aa}
Suppose $C_2\gg M^{-\varepsilon}$ then we have
$$
\mathfrak{D}\ll (NP)^{1/2}M^{1/2+9\theta+\varepsilon}.
$$
\end{lem}

\begin{proof}
Using the Weil bound we get
\begin{align*}
\Omega(\dots)\ll &\mathop{\sum\sum}_{\substack{r'\sim \tilde{R}/c_1\\n\ll \mathfrak{N}}}\:|\mathfrak{c}_{q_1}(n)|\:|\lambda(c_1m'',q_2n)|\:\sum_{\substack{P<p<2P\\p\;\text{prime}}}\;\sum_{\ell\in\mathcal{L}}\mathop{\sum}_{\substack{c_2\in \mathcal{S}}} \;(pc_2-\tilde{q}_2,\;m''n;r'\ell)(r'\ell)^{1/2}.
\end{align*}
(Note that $\tilde{R}<M$, and we are assuming that $M$ is a prime.) Summing over $r'$ and $\ell$ we get
\begin{align*}
\Omega(\dots)\ll &\frac{(\tilde{R}L)^{3/2}PC_2}{c_1^{3/2}}\mathop{\sum}_{\substack{n\ll \mathfrak{N}}}\:|\mathfrak{c}_{q_1}(n)|\:|\lambda(c_1m'',q_2n)|.
\end{align*}
Consequently we get
\begin{align*}
\mathfrak{D}\ll &\frac{N^2(\tilde{R}L)^{1/2}}{M^2P^3}\;\mathop{\sum\sum}_{\substack{m\sim\mathfrak{m}\\d\sim D}}\:\mathop{\sum\sum}_{c_1m''|dm}\:\mathop{\sum\sum}_{\substack{q_1q_2=dm/m''\\\tilde{q}_2\sim Q}}\:\frac{m''q_2}{c_1^{1/2}dm}\;\mathop{\sum}_{\substack{n\ll \mathfrak{N}}}\:|\mathfrak{c}_{q_1}(n)|\:|\lambda(c_1m'',q_2n)|.
\end{align*}
Using Hecke relation we get
\begin{align*}
\mathfrak{D}\ll &\frac{N^2(\tilde{R}L)^{1/2}}{M^2P^3}\;\mathop{\sum\sum\sum\sum}_{\substack{hq_2q_3m''\sim\mathfrak{m}D}}\:\mathop{\sum}_{c_1|hq_2q_3}\:\frac{m''^{1/2}}{q_3}\;\mathop{\sum}_{\substack{n\ll \mathfrak{N}/h}}\:|\lambda(1,hq_2n)|.
\end{align*}
Now taking dyadic subdivision for each variables, and then gluing $hq_2n$ we get
\begin{align*}
\mathfrak{D}\;\lhd\; &\frac{N^2(\tilde{R}L)^{1/2}}{M^2P^3}\;\mathop{\sum\sum}_{\substack{q_3m''\sim\mathfrak{m}D/HQ_2}}\;\frac{m''^{1/2}}{q_3}\;\mathop{\sum}_{\substack{u\ll U}}\:|\lambda(1,u)|.
\end{align*}
where
$$
U=\frac{(L\tilde{R})^{3/2}\tilde{N}^{1/2}M^\varepsilon}{m''^2M^{3/2}}.
$$
Using the Ramanujan bound on average we get
\begin{align*}
\mathfrak{D}\;\ll\; &M^\varepsilon\frac{N^2(\tilde{R}L)^2\tilde{N}^{1/2}}{M^{7/2}P^3}.
\end{align*}
The lemma follows.
\end{proof}

On the other hand if $C_2\leq 1$ the Weil bound yields  that the expression in \eqref{dual-od-red-tran-b} is dominated by
\begin{align*}
O\left( \frac{N^{3/2}P^{3/2}}{M}\right).
\end{align*}
So one still needs to save $P^{3/2}/M^{1/4}$. This can be achieved by using Cauchy and then applying the Poisson summation formula on the sum over $n$. \\

\begin{lem}
\label{lemma-bb}
If $C_2\ll M^{-\varepsilon}$, so that the set $\mathcal{S}$ is at most singleton, then we have
\begin{align*}
\mathfrak{D}\ll M^\varepsilon\:\frac{N^{9/4}(\tilde{R}L)^{1/2}}{M^{5/2}P^{9/4}}\left(L^{3/4}+\frac{(\tilde{R}L)^{1/2}P^{1/4}}{N^{1/4}}\right).
\end{align*}
\end{lem}

\begin{proof}
Suppose $\mathcal{S}$ is singleton with one element $c_2$.
Applying the Cauchy inequality we obtain
\begin{align*}
\Omega(\dots)\ll &\sum_{h|q_1}\;h\;\mathfrak{V}^{1/2}\:\mathfrak{U}^{1/2}
\end{align*}
where
\begin{align*}
\mathfrak{V}=&\mathop{\sum\sum}_{\substack{r'\sim \tilde{R}/c_1\\n\ll \mathfrak{N}/h}}\:|\lambda(c_1m'',hq_2n)|^2
\end{align*}
and
\begin{align*}
\mathfrak{U}=&\mathop{\sum\sum}_{\substack{r'\sim \tilde{R}/c_1\\n\ll \mathfrak{N}/h}}\:\Bigl|\sum_{\substack{P<p<2P\\p\;\text{prime}}}\;\sum_{\ell\in\mathcal{L}} \chi(p\bar{\ell})\;S(c_2-\tilde{q}_2\bar{p},\:\bar{q}_1\bar{\tilde{q}}_2 m''Mhn;r'\ell)\Bigr|^2.
\end{align*}
To $\mathfrak{U}$ we open the absolute square and apply the Poisson summation formula on the sum over $n$ with modulus $r'\ell\ell'$. This yields
\begin{align*}
\mathfrak{U}\ll &\frac{c_1\mathfrak{N}}{hL^2\tilde{R}}\mathop{\sum}_{\substack{r'\sim \tilde{R}/c_1}}\:\mathop{\sum\sum}_{\substack{P<p,p'<2P\\p,p'\;\text{prime}}}\;\mathop{\sum\sum}_{\ell,\ell'\in\mathcal{L}} \;\sum_{|n|\ll hL^2\tilde{R}/c_1\mathfrak{N}} \;|\mathfrak{C}_4| + M^{-2016} 
\end{align*}
where the character sum is given by
\begin{align*}
\mathfrak{C}_4=\sum_{a\bmod{r'\ell\ell'}}S(c_2-\tilde{q}_2\bar{p},\:\bar{q}_1\bar{\tilde{q}}_2 m''Mha;r'\ell)S(c_2-\tilde{q}_2\bar{p}',\:\bar{q}_1\bar{\tilde{q}}_2 m''Mha;r'\ell')e\left(\frac{an}{r'\ell\ell'}\right).
\end{align*}
This character sum has also appeared in \cite{Mu4}, where we proved square root cancellation in the sum using Deligne's result. It follows that
\begin{align*}
\mathfrak{U}\ll &\frac{\tilde{R}^{5/2}P^2L^4}{c_1^{5/2}} +\frac{\mathfrak{N}\tilde{R}^2L^2}{hc_1^2}\left(P+\frac{P^2c_1}{\tilde{R}}\right)
\ll M^\varepsilon\:\frac{(\tilde{R}L)^{5/2}P^2}{c_1^{5/2}}\left(L^{3/2}+\frac{\tilde{R}LP^{1/2}}{N^{1/2}}\right)
\end{align*}
where in the last inequality we assumed that $L\ll M^2/N$.
Consequently we get the bound
\begin{align*}
\Omega(\dots)\ll &M^\varepsilon\:\frac{(\tilde{R}L)^{5/4}P}{c_1^{5/4}}\left(L^{3/4}+\frac{(\tilde{R}L)^{1/2}P^{1/4}}{N^{1/4}}\right)\sum_{h|q_1}\;h\;\mathfrak{V}^{1/2}
\end{align*}
and hence
\begin{align*}
\mathfrak{D}\ll &\frac{N^2CM^\varepsilon}{M^2P^3(L\tilde{R})^{3/4}}\;\left(L^{3/4}+\frac{(\tilde{R}L)^{1/2}P^{1/4}}{N^{1/4}}\right)\\
&\times \mathop{\sum\sum}_{\substack{m\sim\mathfrak{m}\\d\sim D}}\:\mathop{\sum\sum}_{c_1m''|dm}\:\mathop{\sum\sum}_{\substack{q_1q_2=dm/m''\\\tilde{q}_2\sim Q}}\:\frac{m''q_2}{c_1^{1/4}d^2m\tilde{q}_2}\;\sum_{h|q_1}\;h\;\mathfrak{V}^{1/2}.
\end{align*}
Consider the second line of the above expression. We glue $dm$ into a single variable $u$ of size $D\mathfrak{m}$, and write $hh'=q_1$, so that $hh'q_2m''=dm=u$. It follows that the second line is dominated by
\begin{align*}
\frac{1}{D}\mathop{\sum}_{\substack{u\sim D\mathfrak{m}}}\:\mathop{\sum\sum}_{c_1m''|u}\:\mathop{\sum\sum\sum}_{\substack{hh'q_2=u/m''}}\:\frac{1}{c_1^{1/4}h'}\;\mathfrak{V}^{1/2}.
\end{align*}
Applying Cauchy we get that this is bounded by
\begin{align*}
\frac{\mathfrak{m}^{1/2}}{D^{1/2}}\left[\mathop{\sum}_{\substack{u\sim D\mathfrak{m}}}\:\mathop{\sum\sum}_{c_1m''|u}\:\mathop{\sum\sum}_{\substack{hq_2|u/m''}}\:\mathfrak{V}\right]^{1/2},
\end{align*}
which is dominated by $O(M^\varepsilon\mathfrak{m}(PL\tilde{R})^{3/4}N^{-1/4})$, by applying the Ramanujan bound on average. The lemma follows.
\end{proof}

From Lemmas~\ref{lemma-aa} and \ref{lemma-bb} we get that
\begin{align*}
\mathfrak{D}\ll M^\varepsilon N^{1/2}\left[P^{1/2}M^{1/2+9\theta}+ \frac{P^{1/4}M^{5/8+17\theta/4}}{L^{1/2}}\left(L^{3/4}+\frac{M^{11\theta/4}P^{3/4}}{M^{1/8}}\right)\right].
\end{align*}
From this we conclude Proposition~\ref{prop2}.\\



\end{document}